\documentclass{amsart}

\usepackage{amssymb,amsmath,array,multirow,makecell,blindtext,amsthm,enumitem,mathtools,amscd,tikz-cd,amsrefs,dsfont,hyperref,url,caption,float,placeins}
\usepackage[font=small, labelfont=bf]{caption}
\hypersetup{
    colorlinks=true,
    linkcolor=blue,
    citecolor=black,
    filecolor=magenta,      
    urlcolor=cyan,
}
\usepackage[all]{xy}
\usepackage{graphicx}
\usetikzlibrary{positioning}
\usepackage{caption}

\usepackage{graphicx}
\graphicspath{ {./images/} }

\theoremstyle{plain}
\newtheorem{theorem}{Theorem}[section]
\newtheorem{corollary}[theorem]{Corollary}
\newtheorem{lemma}[theorem]{Lemma}
\newtheorem{remark}[theorem]{Remark}
\newtheorem{proposition}[theorem]{Proposition}
\newtheorem{definition}[theorem]{Definition}

\newtheorem{example}[theorem]{Example}
\newcommand{\Z}{\mathbb{Z}}
\newcommand{\Q}{\mathbb{Q}}
\newcommand{\R}{\mathbb{R}}

\newcommand{\LCH}{\operatorname{LCH}}

\newcommand{\ord}{\operatorname{ord}}

\newcommand{\NP}{\operatorname{NP}}

\newcommand{\nf}{\normalfont}

\newcommand{\e}{\varepsilon}

\newcolumntype{?}{!{\vrule width 1pt}}

\title[]{Stretching Newton polygons using pure polynomials}

\author{Rylan Gajek-Leonard and Uri Tomer\\ %\today}
}

\address[]{Department of Mathematics, Union College, Schenectady, NY}
\email[Rylan Gajek-Leonard]{gajekler@union.edu}

\address[]{Department of Mathematics, Union College, Schenectady, NY}
\email[Uri Tomer]{tomeru@union.edu}


\begin{document}

\begin{abstract}
The $p$-adic Newton polygon is a visual tool that encodes information about the roots and factorization of a polynomial relative to a prime $p$.  In this article, we investigate how the Newton polygon changes under polynomial composition. If $f$ and $g$ are polynomials with rational (or $p$-adic) coefficients and the Newton polygon of $g$ is pure (has only one segment), we show under some mild conditions that the Newton polygon of $f\circ g$ is the same as that of $f$, but stretched horizontally by $\deg g$. When $f=g$, this implies that all iterates of certain pure polynomials are irreducible, recovering a classical result of Robert Odoni on the irreducibility of iterated Eisenstein polynomials.
\end{abstract}

\maketitle
%\tableofcontents

%%%%%%%%
\section{Introduction.}

Fix a prime $p$ and a polynomial $f\in \Q[x]$ with $f(0)\neq0$. The $p$-adic Newton polygon of $f$ encodes information about the factorization of $f$ over the field $\Q_p$ of $p$-adic numbers and the valuation of its roots over $\bar \Q_p$. This  local data can often be pieced together at several primes to make global conclusions about the irreducibility and Galois properties of $f$ over $\Q$ (e.g., see \cite{Coleman}), making the Newton polygon a useful tool for studying arithmetic properties of $f$. 

In this article, we consider the behavior of the Newton polygon under composition. Specifically, what can be said about the Newton polygon of $f\circ g$ as $g$ varies?  Can one describe families of polynomials $S$ that preserve, in some sense, the key features (segments and slopes) of the Newton polygon of $f$ under composition? The case $S=\{f\}$ lies at the heart of arithmetic dynamics, whose central focus is understanding the behavior of $f$ upon iteration.  When $f$ is \emph{$p^r$-pure}  (i.e., its Newton polygon is a single line segment connecting the points $(0,r)$ and $(\deg f, 0)$) and $S$ consists of $p^r$-pure polynomials, it is known (see \cite[Theorem 3.10]{DS} and \cite{Ali}) that the Newton polygon of $f\circ g$ is also $p^r$-pure, assuming $\deg f>1$. Thus, the overall shape of a pure Newton polygon is preserved under composition by pure polynomials of the same slope. 

Our main result shows that a similar phenomenon holds when the Newton polygon of $f$ has multiple segments. Roughly speaking, as long as the Newton polygon of $f$ does not have any segments with steep negative slopes, composition with a $p^r$-pure polynomial $g$ has the interesting effect of `stretching' the Newton polygon of $f$ horizontally by a factor of $\deg g$, thereby preserving the number of segments and transforming the slopes in a predictable way (see Figure \ref{mainfig}).

%%%%
\begin{theorem}\label{main} Suppose that the $p$-adic Newton polygon of $f\in \Q_p[x]$ consists of $N$ segments with slopes $s_1< \dots <s_N$. If $g$ is $p^r$-pure and $r+s_i>0$ for all $1\leq i\leq N$ then the Newton polygon of $f\circ g$ has $N$ segments of slopes $\frac{s_1}{\deg g}< \dots <\frac{s_N}{\deg g}$.
\end{theorem}
     
\begin{figure}[h!]
\begin{center}
\hspace{-8cm}
\begin{minipage}{.35\textwidth}
    \begin{tikzpicture}[scale=0.65] 
    \draw[{black}] plot coordinates {(0,4.5) (0,0) (5.7,0)};
     \draw[thick] (0,4) -- (1,1);
     \draw[thick] (1,1) -- (2,.5);
     \draw[thick] (2,.5) -- (2.5,.4);
     \draw[thick] (3.5,.4) -- (4,.5);
     \draw[thick] (4,.5) -- (5,1.3);
      \draw[thick] (5,1.3) -- (5.5,3);
     \node[align=center] at (3,.4) {$\mathbf{\cdots}$};
     \node[align=center] at (0.9,2.7) {\scriptsize$s_1$};
     \node[align=center] at (1.7,1.1) {\scriptsize$s_2$};
      \node[align=center] at (4,1.1) {\scriptsize$s_{N-1}$};
      \node[align=center] at (4.7,2.2) {\scriptsize$s_{N}$};
       \node[align=center] at (3,3) {\scriptsize$\NP_p(f)$};
%\fill (1.667,1.667) circle (2pt);
%\fill (0,2) circle (2pt) node[left] {\scriptsize\text{$2$}};
%\fill (0,1) circle (2pt) node[left] {\scriptsize\text{$1$}};
\fill (0,4) circle (2pt);
\fill (1,1)  circle (2pt);
\fill (2,.5) circle (2pt);
\fill (4,.5)  circle (2pt);
\fill (5,1.3) circle (2pt);
\fill (5.5,3) circle (2pt);
\draw (1, 2pt) -- (1,-2pt)node[anchor=north] {\scriptsize\text{$i_1$}};
\draw (2, 2pt) -- (2,-2pt)node[anchor=north] {\scriptsize\text{$i_2$}};
%\draw (4, 2pt) -- (4,-2pt)node[anchor=north] {\scriptsize\text{$i_{N-2}$}};
%\draw (5, 2pt) -- (5,-2pt)node[anchor=north] {\scriptsize\text{$i_{N-1}$}};
\draw (5.5, 2pt) -- (5.5,-2pt)node[anchor=north] {\scriptsize\text{$n$}};
    \end{tikzpicture}
    \end{minipage}
    \begin{minipage}{0.\textwidth}
    \begin{tikzpicture}[scale=0.65] 
    \draw[{black}] plot coordinates {(0,4.5) (0,0) (11.2,0)};
     \draw[thick] (0,4) -- (2,1);
     \draw[thick] (2,1) -- (4,.5);
     \draw[thick] (4,.5) -- (5,.4);
     \draw[thick] (7,.4) -- (8,.5);
     \draw[thick] (8,.5) -- (10,1.3);
      \draw[thick] (10,1.3) -- (11,3);
      \draw (2, 2pt) -- (2,-2pt)node[anchor=north] {\scriptsize\text{$di_1$}};
\draw (4, 2pt) -- (4,-2pt)node[anchor=north] {\scriptsize\text{$di_2$}};
%\draw (4, 2pt) -- (4,-2pt)node[anchor=north] {\scriptsize\text{$i_{N-2}$}};
%\draw (5, 2pt) -- (5,-2pt)node[anchor=north] {\scriptsize\text{$i_{N-1}$}};
\draw (11, 2pt) -- (11,-2pt)node[anchor=north] {\scriptsize\text{$dn$}};
      \node[align=center] at (6,.4) {$\mathbf{\cdots}$};
     \node[align=center] at (1.5,2.7) {\scriptsize$\frac{s_1}{d}$};
     \node[align=center] at (3,1.2) {\scriptsize$\frac{s_2}{d}$};
      \node[align=center] at (8.7,1.3) {\scriptsize$\frac{s_{N-1}}{d}$};
      \node[align=center] at (10,2.2) {\scriptsize$\frac{s_N}{d}$};
       \node[align=center] at (6,3) {\scriptsize$\NP_p(f\circ g)$};
%\fill (1.667,1.667) circle (2pt);
%\fill (0,2) circle (2pt) node[left] {\scriptsize\text{$2$}};
%\fill (0,1) circle (2pt) node[left] {\scriptsize\text{$1$}};
\fill (0,4) circle (2pt);
\fill (2,1)  circle (2pt);
\fill (4,.5)  circle (2pt);
%\fill (7,.4) circle (2pt);
\fill (8,.5) circle (2pt);
\fill (10,1.3) circle (2pt);
\fill (11,3) circle (2pt);
%\draw (1, 2pt) -- (1,-2pt)node[anchor=north] {\scriptsize\text{$1$}};
%\draw (2, 2pt) -- (2,-2pt)node[anchor=north] {\scriptsize\text{$2$}};
    \end{tikzpicture}
    \end{minipage}
    \end{center}
    \caption{Newton polygons of $f$ and $f\circ g$ as in Theorem \ref{main}, where $d=\deg g$ and $i_j$ denote the $x$-coordinates of the vertices of $\NP_p(f)$.}\label{mainfig}
    \end{figure}
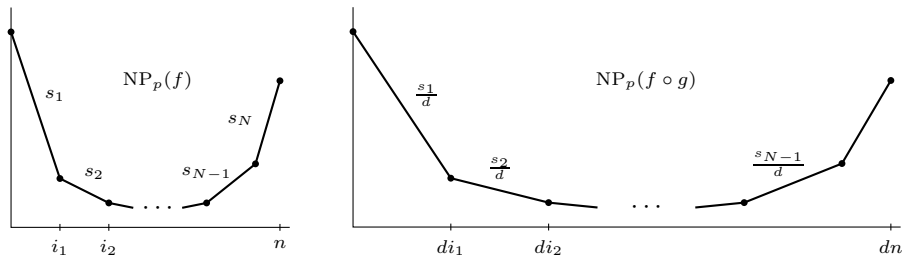
    
 \vspace{-2mm}

\begin{remark}\nf Note that for any fixed $f$, one can always choose $g$ so that the assumptions of Theorem \ref{main} are satisfied. In particular, the conditions are satisfied for any $f$ whose Newton polygon has all positive slopes. 
\end{remark}

Recall the Eisenstein-Dumas criterion \cite{Dumas,Eisenstein}  (see  \cite{Gao} for a modern statement), which says that if $f$ is pure of slope $-\frac{r}{\deg f}$ and $\gcd(r,\deg f)=1$, then $f$ is irreducible over $\Q$. (When $r=1$, this recovers the well-known Eisenstein irreducibility criterion.) Polynomials satisfying this criterion are called \emph{$p^r$-Dumas} polynomials, or $p$-Eisenstein polynomials when $r=1$. 
The family of $p^r$-Dumas polynomials arise naturally in the theory of local fields; for example, they can be used to generate ramified  (totally, in the Eisenstein case) extensions. An immediate consequence of Theorem \ref{main} is the following corollary, which says that $p^r$-Dumas polynomials are \emph{dynamically irreducible} (or \emph{stable}), meaning that all their iterates are also irreducible. Note that this recovers the classical result of Odoni \cite[Lemma 1.3]{Odoni} asserting the dynamical irreducibility of Eisenstein polynomials.

\begin{corollary}\label{ds} If $f$ is $p^r$-Dumas then $f$ is dynamically irreducible.
\end{corollary}
\begin{proof} The case $\deg f=1$ is clear, so assume $\deg f>1$. By Theorem \ref{main}, the Newton polygon of $f\circ f$ has a single segment of slope $-r/(\deg f)^2$, and is therefore irreducible by the Eisenstein-Dumas criterion. The result now follows from induction. 
\end{proof}

\begin{remark} \nf Corollary \ref{ds} was first proved in \cite[Corollary 4.1]{DS} using different methods. 
\end{remark}

As a sample application, we show in Corollary \ref{irredcoro} how one can use Theorem \ref{main} to prove that the sequence 
$$
f_n(x)=1+x+\frac{x^2}{2!}+\cdots+ \frac{x^n}{n!}
$$ 
of Taylor polynomials of the exponential function (all of which are irreducible by \cite{Schur,Coleman}) are also irreducible upon composition with all iterates of certain $p^r$-Dumas polynomials.

\subsection{Acknowledgements}
The authors would like to thank John Cullinan, Daniel Gallagher, Jeffrey Hatley, and James Upton.

\subsection{Notation} We write $f^n=f \cdots f$ for the product of $f$ with itself $n$ times, and $f^{\circ n}=f\circ \dots \circ f$ for the $n$th iterate of $f$.

%%%%%%%%
\section{Newton polygons of products and sums.}

Fix a prime $p$ and a polynomial $f=\sum_{i=0}^n a_ix^i\in \Q_p[x]$. The \emph{$p$-adic Newton polygon of $f$}, denoted $\NP_p(f)$, is the lower convex hull in $\R^2$ of the points $(i,\ord_p a_i)$, $0\leq i\leq n$, where $\ord_p$ denotes the $p$-adic valuation. Thus, the Newton polygon is a subset of $\R^2$ bounded below by a sequence of line segments of increasing slopes and on either side by vertical lines. (The Newton polygon of a monomial $ax^i$ is therefore a vertical ray extending from the point $(i,\ord_p a)$.)
By the \emph{length} of a segment of the Newton polygon, we mean the length of the segment upon projection to the $x$-axis.
 We say that $f$ is \emph{pure at $p$} if its Newton polygon has exactly one segment. Following \cite{Jak,DS}, we say that $f$ is \emph{$p^r$-pure} for an integer $r\geq 1$ if $f$ is pure, $\ord_p(a_0)=r$, and $\ord_p(a_n)=0$. 
 
  \begin{remark} \nf Note the distinction between a polynomial that is pure at $p$ (one segment, arbitrary slope) and $p$-pure (one segment, slope $\frac{-1}{\text{degree}}$). 
 \end{remark}

\begin{example} \nf The Newton polygon of $p+x^2+p^2x^3+p^3x^6$ has vertices at $(0,1)$, $(2,0)$, and $(6,3)$. It therefore has two segments of lengths 2 and 4 with slopes $-\frac{1}{2}$ and $\frac{3}{4}$, respectively. 

\vspace{-2mm}

\begin{figure}[h!]
\centering
\begin{tikzpicture}%[scale=0.8] 
\draw[{black}] plot coordinates {(0,1.7) (0,0) (3.2,0)};
\draw[] (0,0.5) -- (1,0);
\draw[] (1,0) -- (3,1.5);
\fill (0,0.5) circle (2pt) node[left] {\scriptsize\text{$1$}};
\fill (1,0) circle (2pt);
\fill (3,1.5) circle (2pt);
\draw (0.5, 2pt) -- (0.5,-2pt)node[anchor=north] {};
\draw (1, 2pt) -- (1,-2pt)node[anchor=north] {\scriptsize\text{$2$}};
\draw (1.5, 2pt) -- (1.5,-2pt)node[anchor=north] {};
\draw (2, 2pt) -- (2,-2pt)node[anchor=north] {};
\draw (2.5, 2pt) -- (2.5,-2pt)node[anchor=north] {};
\draw (3, 2pt) -- (3,-2pt)node[anchor=north] {\scriptsize\text{$6$}};
\draw (2pt,1) -- (-2pt,1) node[anchor=east] {};
\draw (2pt,1.5) -- (-2pt,1.5) node[anchor=east] {\scriptsize\text{$3$}};
    \end{tikzpicture}
\end{figure}
\end{example}

\vspace{-4mm}

\noindent We have the following fundamental property of Newton polygons. 

\begin{theorem}\label{NPtheorem}
\color{white}SPACE\color{black}
\begin{enumerate}
\item If  $\NP_p(f)$ consists of $N$ segments of lengths $\ell_1,\dots,\ell_N$ then there exist pure $f_1,\dots, f_N\in \Q_p[x]$ with $\deg(f_i)=\ell_i$ such that
$
f=\prod_{i=1}^Nf_i.
$
\item If $\NP_p(f)$ has a segment of length $\ell$ and slope $s$ then $f$ has exactly $\ell$ roots (in $\bar \Q_p$) of valuation $-s$. 
\end{enumerate} 
\end{theorem}
\begin{proof} See \cite[Proposition II.6.3]{Neukirch}.
\end{proof}

An immediate consequence of this theorem is that the Newton polygon of the product of two polynomials can be obtained by concatenating the segments of their respective Newton polygons in order of increasing slope. In the case of pure polynomials, this yields the following lemma. 

\begin{lemma} \label{concat} Suppose $f_1, \dots, f_n\in \Q_p[x]$ are pure of slopes $s_1\leq s_2\leq \cdots \leq s_n$. Then the Newton polygon of $f=\prod_{i=1}^n f_i$ consists of $n$ segments of slopes $s_1, s_2,\dots, s_n$. 
\end{lemma}

\begin{example}\nf The Newton polygon of $(x^2-p)(x^3-p)(x^4-p)$ consists of 3 segments of slopes $-1/2$, $-1/3$, and $-1/4$. The Newton polygon of $(x^2-p)^2$ consists of one segment of slope $-1/2$. 
\end{example}

\begin{corollary}\label{eisenpower} If $f$ is $p^r$-pure then $f^n$ is pure of slope $\frac{-r}{\deg f}$ for all $n\geq 1$. 
\end{corollary}
\begin{proof} This follows from Lemma \ref{concat}.
\end{proof}

We now show that the Newton polygon of a sum of polynomials is bounded below by the lower convex hull of the union of the constituent Newton polygons. In what follows, we write $\LCH(X)$ for the lower convex hull of a subset $X$ of $\R^2$. 

\begin{lemma}\label{sumunion} Let $f_1, \dots, f_n\in \Q_p[x]$. Then 
$$
\NP_p\bigg(\sum_{i=1}^n f_i\bigg)\subseteq \LCH\bigg(\bigcup_{i=1}^n\NP_p(f_i)\bigg).
$$
\end{lemma}
\begin{proof} It suffices to prove the case $n=2$, the rest follows by induction. Let $d$ be the maximum of the degrees of $f_1$ and $f_2$ write $f_1=\sum_{i=1}^d a_ix^i$ and $f_2=\sum_{i=1}^d b_ix^i$. Then Lemma \ref{convexlemma} below implies
\begin{align*}
\NP_p(f_1+f_2)&=\LCH\{\big(i,\ord_p(a_i+b_i)\big)\mid 0\leq i\leq N\}\\
&\subseteq \LCH\{\big(i,\min(\ord_p a_i,\ord_pb_i)\big)\mid 0\leq i\leq N\} \\
&\subseteq \LCH\big(\NP_p(f_1)\cup \NP_p(f_2)\big). 
\end{align*}
\end{proof}

\begin{example} \nf Let $f(x)=3+x^2+9x^3$ and $g(x)=9+x+3x^3$. In this case, 
$$
\NP_3(f+g)= \LCH\big(\NP_3(f)\cup \NP_3(g)\big),
$$
as seen in the following diagram, where the thin, dashed, and bold lines indicate $\NP_3(f)$, $\NP_3(g)$, and the lower convex hull of their union, respectively. 
%\begin{center}
%\includegraphics[scale=0.1]{fig3}   
%\end{center}
%%%%%%%FIGURE 3 TIKZ %%%%%%
 
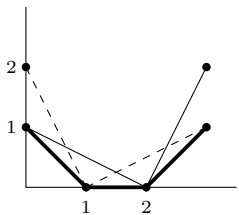
\begin{figure}[h!]
\centering
    \begin{tikzpicture}[scale=0.9] 
    \draw[{black}] plot coordinates {(0,2.5) (0,0) (3.5,0)};
   \draw[line width=0.5mm] (0,1) -- (1,0);
    \draw[line width=0.5mm] (1,0) -- (2,0);
    \draw[line width=0.5mm] (2,0) -- (3,1);
     \draw[] (0,1) -- (2,0);
     \draw[] (2,0) -- (3,2);
      \draw[dashed] (0,2) -- (1,0);
     \draw[dashed] (1,0) -- (3,1);
\fill (0,2) circle (2pt) node[left] {\scriptsize\text{$2$}};
\fill (0,1) circle (2pt) node[left] {\scriptsize\text{$1$}};
\fill (1,0) circle (2pt);
\fill (2,0) circle (2pt);
\fill (3,1) circle (2pt);
\fill (3,2) circle (2pt);
\draw (1, 2pt) -- (1,-2pt)node[anchor=north] {\scriptsize\text{$1$}};
\draw (2, 2pt) -- (2,-2pt)node[anchor=north] {\scriptsize\text{$2$}};
    \end{tikzpicture}
    \end{figure}
\end{example}

\vspace{-7mm}

\begin{lemma} \label{convexlemma} Let $x_1,\dots, x_n,y_1,\dots, y_n\in \R\cup \{\infty\}$, $z_i=\min(x_i,y_i)$, and suppose $X=\{(i,x_i)\mid 1\leq i \leq n\}$, $Y=\{(i,y_i)\mid 1\leq i \leq n\}$, and 
$Z=\{(i,z_i)\mid 1\leq i \leq n\}.$ Then the following hold:
\begin{enumerate}
\item If $x_i\geq y_i$ then $\LCH(X)\subseteq \LCH(Y)$. 
\item $\LCH(Z)\subseteq \LCH(X\cup Y)$. 
\end{enumerate}
\end{lemma}
\begin{proof}

1. By assumption, we have $(i,x_i)\in \LCH(Y)$ for all $i$. Convexity now implies that the line connecting any two points in $X$ is also contained in $\LCH(Y)$. In particular, the edges of $\LCH(X)$ are contained in $\LCH(Y)$ and the result follows. 

2. Clearly $(i,z_i)\in X\cup Y$, hence $(i,z_i)\in \LCH(X\cup Y)$ for all $i$. The result now follows from the same argument as (1). 
\end{proof}

%%%%%%%%
\section{Main result.}\label{mainsec}

Fix a polynomial $f=\sum_{i=0}^n a_i x_i\in \Q_p[x]$ with $f(0)\neq 0$ and an integer $r\geq 1$. We begin by showing that certain pure Newton polygons are preserved under composition with $p^r$-pure polynomials. 
 
 \begin{proposition}\label{main1}  
If $f$ is pure of slope $s$ and $g\in \Z_p[x]$ is $p^r$-pure  with $r+s>0$, then $f\circ g$ is pure of slope $\frac{s}{\deg g}$. 
\end{proposition}

\begin{proof} Since multiplication by nonzero constants only shifts the Newton polygon vertically, we may assume that $f$ is monic. Write
$$
f(g(x))=c_0+c_1x+\cdots +c_{dn}x^{dn},
$$ 
where $d=\deg g$ and $n=\deg f$. Let  $m=\ord_pa_0$, so that $s=-\frac{m}{n}$ (possibly not in lowest terms).  
By assumption, the Newton polygon of $f$ consists of a single edge of slope $s$ connecting the two vertices $(0,m)$ and $(n,0)$, i.e., we have that:

\vspace{0.2cm}

\noindent (V$f$) $\ord_p a_0=m$ and $\ord_p a_{n}=0$,

\noindent (E$f$) $\ord_p(a_i)\geq m+si$ for all $1\leq i\leq n$.

\vspace{0.2cm}

We must prove that $\NP_p(f\circ g)$ consists of one edge of slope $s/d$ connecting the two vertices $(0,m)$ and $(dn,0)$, i.e., we need to show that:

\vspace{0.2cm}

\noindent (V$f\circ g$) $\ord_p(c_0)=m$ and $\ord_p(c_{dn})=0$,

\noindent(E$f\circ g$) $\ord_p(c_i)\geq m+\frac{s}{d}i$ for all $1\leq i\leq dn$.  

\vspace{0.2cm}

Let $b_i\in \Z_p$ denote the coefficients of $g$. To show (V$f\circ g$), note that (E$f$) and $r+s>0$ imply 
\begin{align*}
\ord_p\sum_{i=1}^{n}  a_ib_0^i\geq \min_{1\leq i\leq n}\bigg(\ord_pa_i+ir\bigg)\geq \min_{1\leq i\leq n}\bigg(m+i(r+s)\bigg)>m,
\end{align*}
therefore
\begin{align*}
\ord_p(c_0)=\ord_p\bigg(\sum_{i=0}^{n}  a_ib_0^i\bigg)
=\min\bigg(\ord_pa_0, \ord_p\sum_{i=1}^{n}  a_ib_0^i\bigg)
=m.
\end{align*}
It is clear that
$
\ord_p(c_{dn})=\ord_p(a_nb_d^n)=\ord_p(a_n)+n\ord_p(b_d)=0.
$

It remains to prove (E$f\circ g$), for which it suffices to show that the points $(i,\ord_p c_i)$ lie on or above the line connecting $(0,m)$ and $(nd,0)$, i.e., that
$$
(i,\ord_p c_i)\in \LCH\{(0,m), (nd,0)\}. 
$$
First, note that (E$f$) implies $\ord_p(a_ig(0)^i)=\ord_pa_i+ir\geq m+i(r+s)$
so the leftmost vertex  of $\NP_p(a_ig^i)$ lies on or above the point $\big(0,m+i(r+s)\big)$. But Corollary \ref{eisenpower} implies that the Newton polygon of each $a_ig^i$ is pure of slope $-\frac{r}{d}$, hence the lower edge of  $\NP_p(a_ig^i)$ lies entirely on or above the line segment $L_i$ through the points $\big(0,m+i(r+s)\big)$ and $(id, m+is)$.  Thus 
\begin{equation}\label{lbL}
\NP_p(a_ig^i)\subseteq \LCH(L_i)
\end{equation}
for all $0\leq i\leq n$. (Note that $L_0$ is simply the point $(0,m)$.) Since  $f\circ g=\sum_{i=0}^n a_i g^i$, it now follows from Lemma \ref{sumunion} (see also Figure \ref{fig2}) that
\iffalse
$$
(i,\ord_p c_i)\in \NP_p(f\circ g)=\NP_p\bigg(\sum_{i=0}^n a_i g^i\bigg)\subseteq \LCH\bigg(\bigcup_{i=0}^n \LCH(L_i)
\bigg)=\LCH\{(0,m),(nd,0)\}. 
$$
\fi
\newpage
\begin{align*}
(i,\ord_p c_i)&\in \NP_p(f\circ g)\\
&=\NP_p\bigg(\sum_{i=0}^n a_i g^i\bigg)\\
&\subseteq  \LCH\bigg(\bigcup_{i=0}^n \LCH(L_i)\bigg)\\
&=\LCH\{(0,m),(nd,0)\}.
\end{align*}
\end{proof}

\begin{figure}[h!] 
%\includegraphics[scale=0.24]{fig4}  
%%%%%TIKZ FIGURE 4 %%%%%%%
    \begin{tikzpicture}[scale=0.8] 
    \draw[{black}] plot coordinates {(0,8.5) (0,0) (10.5,0)};
    \newcommand{\pathone}{(0,8) -- (10,0)}
    \draw[thick] \pathone;
    \fill (0,8) circle (2pt) node[left] {\scriptsize $nr$};
    \fill (10,0) circle (2pt) node[below] {\scriptsize $nd$};
    \fill (0,2) circle (2pt) node[left] {\scriptsize $m$};
    \fill (0,6) circle (2pt) node[left] {\scriptsize $m+4(r+s)$};
    \fill (0,5) circle (2pt) node[left] {\scriptsize $m+3(r+s)$};
    \fill (0,4) circle (2pt) node[left] {\scriptsize $m+2(r+s)$};
    \fill (0,3) circle (2pt) node[left] {\scriptsize $m+r+s$};
    \newcommand{\pathtwo}{(0,2) -- (10,0)}
    \draw[thick, dashed] \pathtwo;
    \draw[thick] (0,6) -- (6.667,0.667);
    \draw[thick] (0,5) -- (5,1);
    \draw[thick] (0,4) -- (3.333,1.333);
    \draw[thick] (0,3) -- (1.667,1.667);
    \fill (6.667,0.667) circle (2pt); 
    \fill (5,1) circle (2pt);
    \fill (3.333,1.333) circle (2pt);
    \fill (1.667,1.667) circle (2pt);
\draw (1.667, 2pt) -- (1.667,-2pt)node[anchor=north] {\scriptsize $d$};
\draw (3.333, 2pt) -- (3.333,-2pt)node[anchor=north] {\scriptsize $2d$};
\draw (5, 2pt) -- (5,-2pt)node[anchor=north] {\scriptsize $3d$};
\draw (6.667, 2pt) -- (6.667,-2pt)node[anchor=north] {\scriptsize $4d$};
    \node[align=center,rotate=90] at (-0.3,7.1) {$\mathbf{\cdots}$};
    \node[align=center] at (8.5,-0.35) {$\mathbf{\cdots}$};
    \node[align=center,rotate=45] at (4,3.8) {$\mathbf{\cdots}$};
    \end{tikzpicture}
    \caption{In bold are the lines $L_i$, which serve as lower bounds (in the sense of \eqref{lbL}) for the Newton polygons of $a_ig^i$. The dashed line is the lower convex hull of all the $L_i$, which is exactly the Newton polygon of $f\circ g$.} 
    \label{fig2}
    \end{figure}
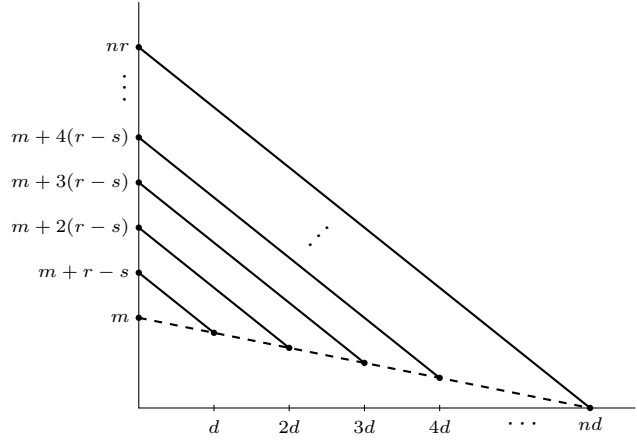

\begin{remark}\label{prcomp} \nf Note that Proposition \ref{main1} implies that if $f$ and $g$ are both $p^r$-pure with $\deg f>1$ then $f\circ g$ is also $p^r$-pure. This result therefore recovers \cite[Theorem 3.10]{DS}, however the proof presented here is different.
 \end{remark}

\begin{remark} \nf Both the $p$-integrality and negative slope of $p^r$-pure polynomials are necessary for Proposition \ref{main1} to hold. For example, one can check that the conclusion of the proposition fails if:
\begin{itemize}
\item $f(x)=g(x)=x^2+\frac{1}{p}$ 
\item $f(x)=p+x^2$ and $g(x)=1+px^2$.
\end{itemize}
Note that in both these examples, $g$ is pure but not $p$-pure.
\end{remark}

We now prove our main theorem.  

\noindent \it Proof of Theorem \ref{main}. \nf 
By Theorem \ref{NPtheorem}, we can factor $f=\prod_{i=1}^N f_i$ so that each $f_i$ is pure of slope $s_i$. By Proposition \ref{main1}, $f_i\circ g$ is pure of slope $\frac{s_i}{\deg g}$. It now follows from Lemma \ref{concat} that the Newton polygon of 
$$
f\circ g=\prod_{i=1}^N f_i\circ g
$$ 
has the desired shape. 
\hfill $\square$

\begin{remark}\nf An earlier version of this article contained a significantly longer inductive proof of this theorem. The `easy' proof presented above was pointed out to us by James Upton, whom we heartily thank. 
\end{remark}

\begin{remark}\nf The assumption on the slopes of $f$ is necessary for Theorem \ref{main}. For example, the conclusion of the theorem is false for $f(x)=p^2+x+p^2x^2$  (two segments of slopes $-2$ and $2$) and the Eisenstein ($p$-pure) polynomial $g(x)=p+x^2$. 
\end{remark}

\begin{corollary}
There are infinitely many irreducible polynomials $g\in \Z_p[x]$ such that the following hold:
\begin{enumerate}
\item The Newton polygon of $f\circ g$ has the same number of segments as the Newton polygon of $f$.
\item For any $\e>0$, all roots (in $\bar \Q_p$) of $f\circ g$ have valuation $<\e$.
\end{enumerate}
\end{corollary}
\begin{proof} Let $s_1, \dots, s_N$ be the slopes of the Newton polygon of $f$. Let $\e>0$ and choose any of the infinitely many pairs of integers $r$ and $d$ such that  $\gcd(r,d)=1$, and both $r+s_i>0$ and $\frac{s_i}{d}<\e$ hold for all $i$. Then $g(x)=x^d+p^r$ is $p^r$-Dumas (i.e., it is irreducible and $p^r$-pure) and statements 1 and 2 now follow directly from Theorem \ref{main} and Theorem \ref{NPtheorem}. 
\end{proof}

%%%%%%
\section{Sample application.}

Recall that a polynomial $g\in \Q[x]$ is called \emph{dynamically irreducible} if $g^{\circ n}$ is irreducible for all $n\geq 0$. We give a mild generalization of this notion in the following definition. 

\begin{definition}\label{dsg} Let $f,g\in \Q[x]$. We say that $g$ is \emph{dynamically irreducible at $f$} if $f\circ g^{\circ n}$ is irreducible for all $n\geq 0$. 
\end{definition}

\noindent  Let
$$
f_n(x)=1+x+\frac{x^2}{2!}+\cdots+\frac{x^n}{n!}\in \Q[x]
$$
denote the $n$th Taylor polynomial of the exponential function. It is well known that $f_n$ is irreducible for all $n$ (see \cite{Schur, Coleman}).

\begin{proposition}\label{irredcoro} Fix $n\geq 1$ and suppose $g\in \Q[x]$ is $p^r$-Dumas at all primes $p\mid n$. If $d=\deg(g)$ is prime and $d$ is strictly larger than all prime divisors of $n$, then $g$ is dynamically irreducible at $f_n$.
\end{proposition}

\begin{proof} Choose a prime $p\mid n$ and fix $m\geq 0$. We must show that $f_n \circ g^{\circ m}$ is irreducible over $\Q$.
Writing $n=\sum_{i=0}^N b_ip^{n_i}$ with $n_1<n_2<\cdots<n_N$ and $0<b_i<p$, we know from \cite[Lemma \S II]{Coleman} that $\NP_p(f_n)$ has $N$ segments of slopes 
$$
-\frac{(p^{n_i}-1)}{p^{n_i}(p-1)}. 
$$
From Proposition \ref{main1} (see also Remark \ref{prcomp}), we know that $g^{\circ m}$ is $p^r$-pure, thus Theorem \ref{main} implies that $\NP_p(f_n\circ g^{\circ m})$ has $N$ segments of  slopes 
$$
-\frac{(p^{n_i}-1)}{d^mp^{n_i}(p-1)}.
$$
In particular, $d^mp^{\ord_pn}=d^mp^{n_1}$ divides the denominator of each slope of the Newton polygon $\NP_p(f_n\circ g^{\circ m})$, thus by \cite[Corollary \S I]{Coleman}, $d^mp^{\ord_pn}$ divides the degree of any irreducible factor of $f_n\circ g^{\circ m}$ over $\Q$. 
It follows that any irreducible factor of $f_n\circ g^{\circ m}$ has degree $d^mn$, hence the result. 
\end{proof}

\begin{remark}\nf  It is easy to construct polynomials $g$ of any prime degree satisfying the hypotheses of Proposition \ref{irredcoro}. For example, if $r$ and $d$ are distinct primes and $\alpha_n=\prod_{\ord_p (n)>0} p$ then
$$
 g(x)=x^d+\alpha_n^r,
$$ 
 is $p^r$-Dumas at all primes $p\mid n$. 
\end{remark}

% \cite{american}

\end{document}